\documentclass[11pt]{article}
\usepackage[T1]{fontenc}
\usepackage{lmodern}
\usepackage{amsmath,amssymb,amsthm,mathtools}
\usepackage[a4paper,margin=28mm]{geometry}
\usepackage{microtype}
\usepackage{enumitem}

\numberwithin{equation}{section}

\newtheorem{theorem}{Theorem}[section]
\newtheorem{lemma}[theorem]{Lemma}

\newtheorem{corollary}[theorem]{Corollary}

\theoremstyle{definition}

\theoremstyle{remark}

\DeclareMathOperator{\tr}{tr}
\DeclareMathOperator{\defect}{def}

\newcommand{\one}{\mathbf 1}
\newcommand{\cD}{\mathcal D}
\newcommand{\cB}{\mathcal B}

\newcommand{\normF}[1]{\left\lVert #1\right\rVert_F}
\newcommand{\ip}[2]{\left\langle #1,#2\right\rangle}

\title{Positive Square Energy of Graphs with\\ Minimum Degree at Least Two}
\author{S. Akbari$^a$\thanks{E-mail address: s\_akbari@sharif.edu}, Fu-Tao Hu$^b$\thanks{E-mail address: hufu@ahu.edu.cn}, Ya-Yang Liu$^b$\thanks{E-mail address: yyl\_math@163.com}\\
{\small $^a$Department of Mathematics, Sharif University of Technology, Tehran, Iran}\\
{\small $^b$Center for Pure Mathematics, School of Mathematical Sciences, Anhui University, Hefei, P.R. China}
}
\date{}

\begin{document}
\maketitle
\vspace{-1.5em}

\begin{abstract}
Let $s^+(G)$ denote the sum of the squares of the positive adjacency
eigenvalues of a graph $G$. 
The square-energy conjecture of Elphick, Farber, Goldberg, and Wocjan,
proved by Liu, Tang, and Zhang, gives a lower bound of $n-1$ for any connected graph of order $n$.
We strengthen this bound to $s^+(G)\ge n$ for every
connected graph $G$ of order $n$ with minimum degree at
least two,  unless $G$ is a cycle.
\end{abstract}


\section{Introduction}
All graphs considered here are finite, simple, and undirected.
For a graph $G$ of order $n$, let $A(G)$ be its adjacency matrix and let
$\lambda_1,\ldots,\lambda_n$ be its adjacency eigenvalues, counted with
multiplicity. The positive and negative square energies are
\[
 s^+(G)=\sum_{\lambda_i>0}\lambda_i^2,
 \qquad
 s^-(G)=\sum_{\lambda_i<0}\lambda_i^2.
\]
They satisfy $s^+(G)+s^-(G)=2|E(G)|$.
Elphick, Farber, Goldberg, and Wocjan~\cite{EFGW} conjectured that both
square energies of a connected graph of order $n$ are at least $n-1$.
Liu, Tang, and Zhang~\cite{LTZ} proved this bound by establishing an
inequality for doubly nonnegative matrices. The bound is attained by
every tree, so a stronger conclusion requires additional structure.

Refinements of the positive square-energy bound were investigated by
Akbari, Kumar, Mohar, Pragada, and Zhang~\cite{AKMPZ}.
Akbari, Hu, and Liu~\cite{AHL} proved that every $2$-connected graph $G$
which is not a cycle satisfies $s^+(G)\ge |V(G)|$.
The cycle formulas recalled below show that this stronger inequality
fails for $C_{4k+1}$, $k\ge1$.
Our main theorem extends the positive square-energy result of~\cite{AHL}
to all connected graphs of minimum degree at least two, allowing cut
vertices and bridges. In this class, cycles are the only graphs that
must be excluded.

\begin{theorem}\label{thm:main}
Let $G$ be a connected graph of order $n$ with $\delta(G)\ge2$.
If $G$ is not a cycle, then
\(
 s^+(G)\ge n.
\)
\end{theorem}

A {\it block} is a maximal $2$-connected subgraph, or a bridge together with
its endpoints. We use $2$-connected only for graphs of order at least
three. A \emph{block graph} is a connected graph all of whose blocks are
complete; a one-vertex graph is also regarded as a block graph.
A \emph{leaf block} is a block containing exactly one cut vertex of the
ambient graph. Its other vertices are called its private vertices.
We denote the set of blocks of $H$ by $\cB(H)$.
For a connected graph $H$, define the {\it surplus}
\begin{equation}\label{eq:rho}
 \rho(H)=s^+(H)-|V(H)|+1.
\end{equation}
Thus Theorem~\ref{thm:main} is equivalent to $\rho(G)\ge1$.

\section{Spectral preliminaries}\label{sec:prelim}
For a real symmetric matrix $A$, write
\[
 A=A_+-A_-,\qquad A_+,A_-\succeq0,\qquad A_+A_-=0.
\]
For an adjacency matrix these are its positive and negative spectral
parts, and $s^\pm(G)=\normF{A_\pm}^2$.
We also write $|A|=A_++A_-=(A^2)^{1/2}$.
A matrix $M$ is \emph{doubly nonnegative} if it is positive semidefinite and
entrywise nonnegative. For a connected graph $H$, put
\begin{equation}\label{eq:qST}
 q(H)=2|E(H)|-|V(H)|+1,\qquad
 S_H(M)=\sum_{uv\in E(H)}\sqrt{M_{uv}},\qquad
 T(M)=\one^{\mathsf T}M\one.
\end{equation}
Every edge in a sum is counted once.
We use the following result of Liu, Tang, and
Zhang~\cite[Theorems~1.2 and~2.1]{LTZ}.
\begin{theorem}[Liu--Tang--Zhang]\label{thm:LTZ}
If $H$ is connected and $M$ is a doubly nonnegative matrix indexed by
$V(H)$, then
\begin{equation}\label{eq:LTZmatrix}
 4S_H(M)^2\le q(H)T(M).
\end{equation}
Moreover,
\begin{equation}\label{eq:connected-bound}
 s^+(H),s^-(H)\ge |V(H)|-1.
\end{equation}
\end{theorem}
In particular, $\rho(H)\ge0$ for every connected graph $H$.

We shall use the following standard Schur complement criterion.

\begin{lemma}[Schur complement criterion {\cite[Theorem~7.7.6]{HJ}}]\label{lem:HJ}
Let
\[
M=
\begin{pmatrix}
P&R\\
R^{\mathsf T}&Q
\end{pmatrix}
\]
be a real symmetric block matrix, where $P$ is positive definite. Then
\[
M\succeq0
\quad\Longleftrightarrow\quad
Q-R^{\mathsf T}P^{-1}R\succeq0.
\]
The matrix $Q-R^{\mathsf T}P^{-1}R$ is called the Schur complement of $P$ in $M$.
\end{lemma}

The variational identities below are equivalent formulations of
Higham's nearest positive semidefinite matrix theorem for the Frobenius
norm~\cite{Higham}. The partition inequality is the $p=2$ case of the
super-additivity theorem of Akbari, Kumar, Mohar, and
Pragada~\cite[Theorem~3]{AKMP}.
\begin{lemma}\label{lem:variational}
For every real symmetric matrix $A$,
\begin{align}
 \normF{A_+}^2
 &=\max_{X\succeq0}\bigl(2\tr(AX)-\normF{X}^2\bigr),
 \label{eq:variational-max}\\
 \normF{A_+}^2
 &=\min_{Y\succeq0}\normF{A+Y}^2.
 \label{eq:variational-min}
\end{align}
Consequently, if $V(G)=U\sqcup W$, then
\begin{equation}\label{eq:partition}
 s^+(G)\ge s^+(G[U])+s^+(G[W]).
\end{equation}
\end{lemma}

For a vertex $v$ of a graph $H$, set
\begin{equation}\label{eq:mu}
 \mu_v(H)=2\bigl(A^2(H)_+\bigr)_{vv}
                -\bigl(A(H)_+\bigr)_{vv}^{\,2}.
\end{equation}
This is the contribution of the row and column indexed by $v$ to
$\normF{A(H)_+}^2$, with the diagonal entry counted once.

The next deletion estimate is a principal-submatrix consequence of
Higham's projection theorem~\cite{Higham}; see also
Zhang~\cite[Section~4, proof of Theorem~1.10]{Zhang} for the corresponding
row-and-column estimate. The connected-case bound additionally
uses~\eqref{eq:connected-bound}. We record both inequalities in the
notation~\eqref{eq:mu}.
\begin{lemma}\label{lem:deletion}
For every graph $H$ and $v\in V(H)$,
\begin{equation}\label{eq:deletion}
 s^+(H)\ge s^+(H-v)+\mu_v(H).
\end{equation}
If $H$ has order at least two and $H-v$ is connected, then
\begin{equation}\label{eq:mu-upper}
 \mu_v(H)\le\rho(H)+1.
\end{equation}
\end{lemma}

\begin{lemma}\label{lem:coalescence}
Let $H_1$ and $H_2$ be disjoint connected graphs, each of order at least two,
and let $H$ be obtained by identifying $v_1\in V(H_1)$ with
$v_2\in V(H_2)$. Then
\begin{equation}\label{eq:coalescence}
 \rho(H)\ge\rho(H_1)+\rho(H_2)
 -\frac29\sqrt{\mu_{v_1}(H_1)\mu_{v_2}(H_2)}.
\end{equation}
If both $H_i-v_i$ are connected, then
\begin{equation}\label{eq:coalescence-linear}
 \rho(H)\ge\frac89\bigl(\rho(H_1)+\rho(H_2)\bigr)-\frac29.
\end{equation}
\end{lemma}
\begin{proof}
With the identified vertices listed first, write
\[
 A(H_1)_+=\begin{pmatrix}p&r^{\mathsf T}\\r&Q_1\end{pmatrix},
 \qquad
 A(H_2)_+=\begin{pmatrix}q&t^{\mathsf T}\\t&Q_2\end{pmatrix}.
\]
Here $p,q>0$. Indeed, a vertex with positive degree has nonzero spectral
weight on both signs, because its first spectral moment is zero and its
second moment is positive. Put
\(
 a=\|r\|^2\) and \( b=\|t\|^2.
\)
Then $\mu_{v_1}(H_1)=p^2+2a$ and $\mu_{v_2}(H_2)=q^2+2b$.
For $\max\{p,q\}\le z\le p+q$, define
\[
 \gamma=\frac1z-
 \sqrt{\left(\frac1p-\frac1z\right)
       \left(\frac1q-\frac1z\right)}.
\]
Consider the symmetric matrix
\[
 Y=\begin{pmatrix}
 z&r^{\mathsf T}&t^{\mathsf T}\\
 r&Q_1&\gamma rt^{\mathsf T}\\
 t&\gamma tr^{\mathsf T}&Q_2
 \end{pmatrix}.
\]
It is positive semidefinite. To see this, use
$Q_1\succeq rr^{\mathsf T}/p$ and
$Q_2\succeq tt^{\mathsf T}/q$, and take the Schur complement of $z$.
Indeed, by the Schur complement criterion~in Lemma~\ref{lem:HJ},
\[
Q_1-\frac{rr^{\mathsf T}}p\succeq0,
\qquad
Q_2-\frac{tt^{\mathsf T}}q\succeq0.
\]
Put
\[
\alpha=\frac1p-\frac1z,\qquad
\beta=\frac1q-\frac1z.
\]
Since $z\ge\max{p,q}$, we have $\alpha,\beta\ge0$, and the definition of $\gamma$ gives
\[
\gamma-\frac1z=-\sqrt{\alpha\beta}.
\]
Hence the Schur complement of the (1,1)-entry $z$ in $Y$ is
\[
\begin{pmatrix}
Q_1-\dfrac{rr^{\mathsf T}}p&0\\
0&Q_2-\dfrac{tt^{\mathsf T}}q
\end{pmatrix}
+
\begin{pmatrix}\sqrt{\alpha}r\\-\sqrt{\beta}t\end{pmatrix}
\begin{pmatrix}\sqrt{\alpha}r\\-\sqrt{\beta}t\end{pmatrix}^{\mathsf T}
\succeq0.
\]
Another application of the Schur complement criterion therefore yields
$Y\succeq0$.

The remaining rank-two part is positive semidefinite because its scalar
coefficient matrix is
\[
 \begin{pmatrix}
 \alpha&-\sqrt{\alpha\beta}\\
 -\sqrt{\alpha\beta}&\beta
 \end{pmatrix},
 \qquad \alpha=\frac1p-\frac1z,\quad
 \beta=\frac1q-\frac1z.
\]
There are no edges between the two sets of private vertices.
Substituting $Y$ into \eqref{eq:variational-max} therefore gives
\begin{equation}\label{eq:Lz-bound}
 s^+(H)\ge s^+(H_1)+s^+(H_2)-L(z),
 \qquad L(z)=z^2-p^2-q^2+2ab\gamma^2.
\end{equation}

We next verify the scalar estimate
\begin{equation}\label{eq:Lz-target}
 \min_{\max\{p,q\}\le z\le p+q}L(z)
 \le\frac29\sqrt{(p^2+2a)(q^2+2b)}.
\end{equation}
Set
\(
 h=\frac{\sqrt{ab}}{pq}\) and \(
 x=\frac{p+q}{\sqrt{pq}}\ge2.
\)
The arithmetic--geometric mean inequality gives
\begin{equation}\label{eq:mu-product}
 \sqrt{(p^2+2a)(q^2+2b)}\ge pq(1+2h).
\end{equation}
If $h\le1/2$, take $z=\max\{p,q\}$. Then
\[
 L(z)=\min\{p^2,q^2\}(2h^2-1)\le0.
\]

Suppose $h>1/2$. For $0<s\le1/3$, put
\(
 \tau=\frac{2s}{x}\) and \(
 z=\sqrt{pq}\,\frac{x-2\tau}{1-\tau^2}.
\)
This $z$ lies in $[\max\{p,q\},p+q]$ and satisfies
$\gamma=\tau/\sqrt{pq}$. For example, if
$\alpha=\sqrt{p/q}\ge1$, the lower endpoint condition is equivalent to
\[
 x-2\tau-\alpha(1-\tau^2)
 =\alpha\left(\tau-\frac1\alpha\right)^2\ge0;
\]
the upper endpoint condition follows from $\tau\le2/x$.
Writing $y=x^2\ge4$, we obtain
\begin{equation}\label{eq:Lz-envelope}
 \frac{L(z)}{pq}
 =y\frac{(y-4s)^2}{(y-4s^2)^2}-y+2+\frac{8h^2s^2}{y}
 \le\frac4{(1+s)^2}-2+2h^2s^2.
\end{equation}
Indeed, the derivative of the first three terms with respect to $y$ is
\[
 \frac{-16s^2(s-1)\bigl(y(3s-1)-4s^2(s+1)\bigr)}{(y-4s^2)^3}
 \le0,
\]
and the final term is also decreasing in $y$.

For $1/2<h\le3/2$, choose $s=1/3$. The right side of
\eqref{eq:Lz-envelope} is $1/4+2h^2/9$, and
\[
 \frac29(1+2h)-\left(\frac14+\frac{2h^2}{9}\right)
 =\frac{7-8(h-1)^2}{36}>0.
\]
For $h\ge3/2$, choose $s=1/(2h)$. In this case
\begin{align*}
 &\frac29(1+2h)
 -\left(\frac4{(1+1/(2h))^2}-\frac32\right)\\
 &\hspace{2em}=
 \frac{32h^3-132h^2+132h+31}{18(2h+1)^2}>0,
\end{align*}
since the numerator equals
\[
 \left(h-\frac94\right)^2(32h+12)+24h-\frac{119}{4}>0
 \qquad(h\ge3/2).
\]
Together with \eqref{eq:mu-product}, these calculations prove
\eqref{eq:Lz-target}. Since
$|V(H)|=|V(H_1)|+|V(H_2)|-1$, \eqref{eq:Lz-bound} gives
\eqref{eq:coalescence}.

Finally, when both vertex deletions are connected, apply
\eqref{eq:mu-upper} and
\[
 \sqrt{(\rho(H_1)+1)(\rho(H_2)+1)}
 \le\frac{\rho(H_1)+\rho(H_2)+2}{2}
\]
to obtain \eqref{eq:coalescence-linear}.
\end{proof}

\section{Main Results}\label{sec:surplus}
\begin{lemma}\label{lem:max-degree}
If $d=\Delta(H)\ge3$ and $d_H(v)=d$, then
\begin{equation}\label{eq:mu-max-degree}
 \mu_v(H)>\frac{19}{9}.
\end{equation}
Consequently, every $2$-connected graph $B$ that is not a cycle satisfies
\begin{equation}\label{eq:block-surplus}
 \rho(B)>\frac{10}{9}.
\end{equation}
\end{lemma}
\begin{proof}
Choose an orthonormal eigenbasis $u_1,\ldots,u_n$ of $A(H)$ and put
$w_j=u_j(v)^2$. The spectral decomposition gives
\begin{equation}\label{eq:moments}
 \sum_jw_j=1,\quad
 \sum_jw_j\lambda_j=0,\quad
 \sum_jw_j\lambda_j^2=d,\quad
 \sum_jw_j\lambda_j^3\ge0,\quad |\lambda_j|\le d.
\end{equation}
The third moment counts twice the number of triangles containing $v$.
Define
\[
 a=\sum_{\lambda_j>0}w_j\lambda_j,
 \qquad c=\sum_{\lambda_j>0}w_j\lambda_j^2,
 \qquad e=\sum_{\lambda_j<0}w_j\lambda_j^2=d-c.
\]
Then $a=(A(H)_+)_{vv}$ and $\mu_v(H)=2c-a^2$.
The positive and negative first absolute moments both equal $a$.
Cauchy--Schwarz on the two signs, followed by the fact that their total
weight is at most one, gives
\begin{equation}\label{eq:moment-cs}
 a^2\le\frac{ce}{d}.
\end{equation}
Also,
\[
 e^2\le
 a\sum_{\lambda_j<0}w_j|\lambda_j|^3
 \le a\sum_{\lambda_j>0}w_j\lambda_j^3
 \le adc.
\]
Combining this with \eqref{eq:moment-cs} yields $e^3\le dc^3$ and hence
\begin{equation}\label{eq:moment-lower}
 c\ge\frac{d}{1+d^{1/3}},\qquad
 \mu_v(H)\ge c+\frac{c^2}{d}.
\end{equation}
For $d\ge4$, the expression
$d(d^{1/3}+2)/(d^{1/3}+1)^2$ is increasing in $d$, so
\[
 \mu_v(H)\ge
 \frac{d(d^{1/3}+2)}{(d^{1/3}+1)^2}
 >\frac{360}{169}>\frac{19}{9}.
\]
For the strict middle inequality at $d=4$, use $4^{1/3}<8/5$.

For $d=3$, use the following inequality, valid for $-3\le t\le3$:
\begin{equation}\label{eq:cubic-minorant}
 t_+^2-\frac8{11}t_+
 \ge\frac4{33}(t+2)^2\left(t-\frac34\right),
 \qquad t_+=\max\{t,0\}.
\end{equation}
For $t\le0$ the right side is nonpositive; for $t\ge0$ the difference
between the two sides is $\frac4{33}(t-1)^2(3-t)$.
Averaging \eqref{eq:cubic-minorant} using \eqref{eq:moments} gives
\(
 c\ge\frac{9+8a}{11}.
\)
Furthermore,
\(
 \frac12\le a\le\frac{\sqrt3}{2}.
\)
Indeed, $3\le3\sum_jw_j|\lambda_j|=6a$, while
\eqref{eq:moment-cs} gives $a^2\le ce/3\le3/4$.
Concavity in $a$ now implies
\begin{align*}
 \mu_v(H)
 &\ge\frac{18+16a}{11}-a^2\\
 &\ge\min\left\{\frac{93}{44},\frac{39+32\sqrt3}{44}\right\}
 =\frac{93}{44}>\frac{19}{9}.
\end{align*}
This proves \eqref{eq:mu-max-degree}.

If $B$ is $2$-connected and is not a cycle, then $\Delta(B)\ge3$.
For a maximum-degree vertex $v$, the graph $B-v$ is connected. By
Lemma~\ref{lem:deletion} and Theorem~\ref{thm:LTZ},
\[
 s^+(B)\ge |V(B)|-2+\mu_v(B)>|V(B)|+\frac19,
\]
which is \eqref{eq:block-surplus}.
\end{proof}

\begin{samepage}
The following explicit cycle formulas are given in~\cite[Proposition~9.1]{AKMPZ}.

\begin{lemma}\label{lem:cycles}
For $k\ge3$,
\begin{equation}\label{eq:cycles}
 s^+(C_k)=
 \begin{cases}
 k,&k\text{ even},\\
 k+1-\sec(\pi/k),&k\equiv1\pmod4,\\
 k-1+\sec(\pi/k),&k\equiv3\pmod4.
 \end{cases}
\end{equation}
\end{lemma}
\end{samepage}

We need the following quantitative consequence of the unicyclic
Coulson identity in Ning and Zeng~\cite[Lemma~2.4 and the proof of
Theorem~1.2]{NZ}.
\begin{lemma}\label{lem:unicyclic}
Every connected unicyclic graph $U$ satisfies
\begin{equation}\label{eq:unicyclic-surplus}
 \rho(U)\ge3-\sqrt5>\frac34.
\end{equation}
\end{lemma}
\begin{proof}
Let $C_k$ be the unique cycle and let $F=U-V(C_k)$. For a graph $J$,
write
\[
 M_J(t)=\sum_{j\ge0}m_j(J)t^{|V(J)|-2j}\qquad(t>0),
\]
where $m_j(J)$ is the number of matchings of size $j$ and
$M_{\varnothing}(t)=1$.
The unicyclic Coulson identity of Ning and Zeng~\cite[Lemma~2.4 and the
proof of Theorem~1.2]{NZ} gives, for odd $k$,
\begin{equation}\label{eq:unicyclic-integral}
 s^+(U)-|V(U)|=
 \begin{cases}
 -I(U),&k\equiv1\pmod4,\\
 I(U),&k\equiv3\pmod4,
 \end{cases}
 \quad
 I(U)=\frac2\pi\int_0^\infty
 t\arctan\!\left(\frac{2M_F(t)}{M_U(t)}\right)\,dt.
\end{equation}
For even $k$, bipartiteness gives $s^+(U)=|V(U)|$.
Thus only $k\equiv1\pmod4$ needs an estimate.
Every matching of the disjoint union $C_k\cup F$ is also a matching of
$U$, so coefficientwise
\[
 M_U(t)\ge M_{C_k}(t)M_F(t).
\]
Since $\arctan$ is increasing, \eqref{eq:unicyclic-integral} and
Lemma~\ref{lem:cycles} imply
\[
 |V(U)|-s^+(U)
 \le k-s^+(C_k)=\sec(\pi/k)-1\le\sqrt5-2.
\]
Here $k\ge5$ and $\sec(\pi/5)=\sqrt5-1$.
The desired bound follows from \eqref{eq:rho}.
\end{proof}

\begin{lemma}\label{lem:pendant-path}
Let $H$ be connected, let $H-v$ be connected, and suppose
$\rho(H)\ge10/9$. Attach a pendant path of arbitrary length at $v$,
obtaining $H'$. Then
\begin{equation}\label{eq:path-surplus}
 \rho(H')>\frac34.
\end{equation}
\end{lemma}
\begin{proof}
The zero-length case is immediate. Otherwise let $P$ be the attached
path before its endpoint $w$ is identified with $v$.
At an endpoint of a nontrivial path, the second spectral moment is one
and the fourth is at most two. H\"older's inequality therefore gives
\(
 |A(P)|_{ww}\ge\frac1{\sqrt2}.
\)
More explicitly, for the endpoint spectral weights,
$$\sum_jw_j\lambda_j^2\le
(\sum_jw_j|\lambda_j|)^{2/3}
(\sum_jw_j\lambda_j^4)^{1/3}.$$
Bipartiteness implies
$(A(P)_+)_{ww}=|A(P)|_{ww}/2$ and
$(A(P)_+^2)_{ww}=1/2$. Thus
\[
 \mu_w(P)=1-(A(P)_+)_{ww}^2\le\frac78,
 \qquad \rho(P)=0.
\]
By Lemmas~\ref{lem:coalescence} and \ref{lem:deletion},
\begin{equation}\label{eq:path-function}
 \rho(H')\ge\rho(H)-\frac29
 \sqrt{\frac78(\rho(H)+1)}.
\end{equation}
The right side is increasing for $\rho(H)\ge10/9$, and at $10/9$ it
is
\[
 \frac{10}{9}-\frac29\sqrt{\frac{133}{72}}
 >\frac79>\frac34.
\]
\end{proof}

\begin{lemma}\label{lem:block-chain}
Let $H$ be obtained from a $2$-connected graph $B$ by successively
attaching complete blocks, including possible bridge blocks $K_2$,
in a chain. Thus the block--cut tree of $H$ is a path with $B$ as an
end block, unless $H=B$. Then
\begin{equation}\label{eq:chain-surplus}
 \rho(H)\ge\frac34.
\end{equation}
\end{lemma}
\begin{proof}
First attach the initial run of bridges. If $B$ is a cycle, the graph
so obtained is unicyclic, and Lemma~\ref{lem:unicyclic} applies.
Otherwise use Lemmas~\ref{lem:max-degree} and \ref{lem:pendant-path}.
For each subsequently attached complete block $K_t$, $t\ge3$,
\[
 \rho(K_t)=(t-1)(t-2)\ge2.
\]
The attachment vertex is a non-cut vertex on each side before
identification. Hence \eqref{eq:coalescence-linear} gives a new surplus
at least
\[
 \frac89\left(\frac34+2\right)-\frac29=\frac{20}{9}.
\]
A following run of bridges is handled by Lemma~\ref{lem:pendant-path}.
The exit vertex in the last block is again a non-cut vertex of the
current graph. Repeating this argument proves the lemma.
\end{proof}

For a connected graph $H$, let $\cD(H)$ be the cone of doubly
nonnegative matrices $M$ indexed by $V(H)$ such that
$M_{uv}=0$ whenever $u\ne v$ and $uv\notin E(H)$. Define
\begin{equation}\label{eq:kappa}
 \kappa(H)=\sup_{0\ne M\in\cD(H)}\frac{4S_H(M)^2}{T(M)},
 \qquad \defect(H)=q(H)-\kappa(H).
\end{equation}
By Theorem~\ref{thm:LTZ}, $\defect(H)\ge0$.
For complete graphs,
\begin{equation}\label{eq:complete-kappa}
 \kappa(K_t)=q(K_t)=(t-1)^2,
\end{equation}
as is seen by taking $M=\one\one^{\mathsf T}$.

The next lemma is the blockwise form of the cut-vertex splitting in
Liu, Tang, and Zhang~\cite[proof of Theorem~2.1]{LTZ}. The inequalities
for $\kappa$ and $\defect$ are its Cauchy--Schwarz consequences, using
$q(H)=\sum_{B\in\cB(H)}q(B)$.
\begin{lemma}\label{lem:matrix-splitting}
Let $H$ be a connected graph of order at least two and
$M\in\cD(H)$. There are matrices $M_B\in\cD(B)$, one for each block
$B\in\cB(H)$, such that
\begin{equation}\label{eq:matrix-splitting}
 M=\sum_{B\in\cB(H)}\widetilde M_B.
\end{equation}
Here a tilde denotes extension by zero to $V(H)$.
In particular, all edge entries are preserved, and
\begin{equation}\label{eq:block-kappa}
 T(M)=\sum_BT(M_B),\qquad
 \kappa(H)\le\sum_B\kappa(B),\qquad
 \defect(H)\ge\sum_B\defect(B).
\end{equation}
The diagonal entry of a vertex which is not a cut vertex of $H$ is
unchanged in its unique block matrix.
\end{lemma}

We also use the doubly nonnegative, weak-duality form of the
correlation-matrix estimate in Liu, Tang, and
Zhang~\cite[Section~1.2, equation~(4) and the weighted estimate]{LTZ}.
In the notation~\eqref{eq:kappa}, it reads as follows.
\begin{lemma}\label{lem:certificate}
Let $C$ be a positive semidefinite matrix indexed by $V(H)$, with
$C_{vv}=1$ for every vertex and $C_{uv}<1$ for every edge. Then
\begin{equation}\label{eq:certificate}
 \kappa(H)\le\Phi_H(C):=\sum_{uv\in E(H)}\frac2{1-C_{uv}}.
\end{equation}
\end{lemma}

We construct such matrices as Gram matrices of unit vectors.
For a clique $K_t$, the regular simplex has pairwise correlations
$-1/(t-1)$, and its cost is
\begin{equation}\label{eq:simplex-cost}
 \Phi_{K_t}=(t-1)^2=q(K_t).
\end{equation}
Any prescribed pair of unit vectors with correlation $-1/(t-1)$ can
be completed to this simplex, by an isometry and additional orthogonal
dimensions.

We will also use a path construction. For a path with $L$ edges and
terminal correlation $c$, put
$\theta=\arccos((-1)^Lc)$. Vectors chosen along a planar arc, with
alternating signs, give every edge correlation
$-\cos(\theta/L)$. Whenever this is less than one, the resulting cost is
\begin{equation}\label{eq:path-certificate}
 \Phi=L\sec^2\frac{\theta}{2L}
 =L+L\tan^2\frac{\theta}{2L}.
\end{equation}
For example, take
$x_j=(-1)^j(\cos(j\theta/L),\sin(j\theta/L))$, $0\le j\le L$.
For fixed $a>0$, the function $L\tan^2(a/L)$ decreases whenever
$0<a/L<\pi/2$; differentiating and using
$\tan z<z\sec^2z$ verifies this assertion.

An \emph{ordinary $u$--$v$ clique-chain} is a block graph with at least
two blocks whose block--cut tree is a path, with $u$ and $v$ non-cut
vertices in distinct end blocks. A path with two edges is included.
For $s\ge2$, put
\(
 T_s=K_{s+2}-uv,
\)
with the two nonadjacent vertices $u,v$ as terminals.

\begin{lemma}\label{lem:terminal-certificates}
Every ordinary $u$--$v$ clique-chain $P$, and every $T_s$, admits a unit
vector certificate with orthogonal terminal vectors and
\begin{equation}\label{eq:terminal-certificate}
 \Phi_P\le q(P)+\eta(P),
\end{equation}
where one may take
\begin{equation}\label{eq:eta-values}
 \eta(P)=
 \begin{cases}
 6-4\sqrt2,&P\text{ is a path with two edges},\\
 2/7,&P\text{ is any other ordinary clique-chain},\\
 \displaystyle\frac{4s}{(s+2)(s^2+2s-1)},&P=T_s.
 \end{cases}
\end{equation}
In particular, every value is at most $6-4\sqrt2$, and every value
except that for the two-edge path is at most $2/7$.
\end{lemma}
\begin{proof}
First suppose $P$ is a path with $L\ge2$ edges. Take $c=0$ in
\eqref{eq:path-certificate}. Its excess over $q(P)=L$ is
$
 L\tan^2\frac{\pi}{4L}.
$
For $L=2$ this is $6-4\sqrt2$; for $L\ge3$ it is at most
$3\tan^2(\pi/12)=21-12\sqrt3<2/7$.

Now suppose there is exactly one clique block of order $t\ge3$.
There is at least one bridge. Give every bridge but one antipodal
endpoint vectors; these bridges merely introduce signs into the
remaining terminal constraints. Write $a=1/(t-1)\le1/2$.
For the remaining clique and bridge choose correlations of absolute
values $a$ and $\sqrt{1-a^2}$, respectively, and give the final terminals
correlation zero. The required three-vector Gram matrix is positive
semidefinite, because it has diagonal entries one and determinant
$1-a^2-(1-a^2)=0$. Choose the actual bridge correlation to be
$-\sqrt{1-a^2}$, absorbing any intervening antipodal signs into the
three prescribed vectors. Complete the clique to its regular simplex.
Only the chosen bridge has excess, and that excess is at most
\[
 \frac2{1+\sqrt3/2}-1=7-4\sqrt3<\frac27.
\]

If there are at least two clique blocks of order at least three, use
regular simplices in the first two such blocks. After incorporating
the intervening antipodal bridges, the two consecutive prescribed
correlations have absolute values at most $1/2$.
They can be realized with the output vector orthogonal to the initial
terminal: the relevant three-vector Gram matrix has the form
\[
 \begin{pmatrix}1&a&0\\a&1&b\\0&b&1\end{pmatrix},
 \qquad |a|,|b|\le\frac12,
\]
and determinant $1-a^2-b^2\ge1/2$.
All subsequent clique simplices and bridges can be placed in the
orthogonal complement of the initial terminal. Thus all blocks have
exactly their costs in \eqref{eq:simplex-cost}, and the excess is zero.
This finishes the ordinary-chain case.

For $T_s$, take orthogonal unit terminal vectors $x_u,x_v$ and a regular
simplex $w_1,\ldots,w_s$ in their common orthogonal complement.
Set $p=1/(s+1)$ and
\[
 z_i=-p(x_u+x_v)+\sqrt{1-2p^2}\,w_i\qquad(1\le i\le s).
\]
Then
\[
 \ip{x_u}{z_i}=\ip{x_v}{z_i}=-\frac1{s+1},\qquad
 \ip{z_i}{z_j}=-\frac{s^2+1}{(s-1)(s+1)^2}\quad(i\ne j).
\]
Since $q(T_s)=s^2+2s-1$, direct summation yields
\begin{align*}
 \Phi_{T_s}
 &=\frac{4s(s+1)}{s+2}
   +\frac{(s-1)^2(s+1)^2}{s^2+2s-1},\\
 \Phi_{T_s}-q(T_s)
 &=\frac{4s}{(s+2)(s^2+2s-1)}\le\frac27.
\end{align*}
The last inequality follows from
$(s+2)(s+2-1/s)\ge14$ for $s\ge2$.
\end{proof}

A \emph{cyclic clique-chain} is obtained from a cycle on distinct
vertices $z_0,\ldots,z_{\ell-1}$, $\ell\ge3$, by replacing each cycle
edge $z_jz_{j+1}$ by a clique $K_{t_j}$ containing its two endpoints;
the additional vertices of the cliques are mutually disjoint.
Indices are read modulo $\ell$ and $t_j\ge2$.
For this graph $B$,
\begin{equation}\label{eq:cyclic-count}
 q(B)=\sum_{j=0}^{\ell-1}(t_j-1)^2+1.
\end{equation}

\begin{lemma}\label{lem:cyclic-certificate}
Every noncomplete cyclic clique-chain $B$ admits a correlation
certificate satisfying
\begin{equation}\label{eq:cyclic-certificate}
 \Phi_B\le q(B)-\frac13.
\end{equation}
\end{lemma}
\begin{proof}
Let $k$ be the number of clique links of order at least three.
A bridge link has minimum cost one, achieved by antipodal vectors.
A nontrivial clique link has the cost in \eqref{eq:simplex-cost}.
In view of \eqref{eq:cyclic-count}, it is enough to obtain total excess
at most $2/3$ over the sum of these link costs.

Suppose first that $k\ge3$. Give all bridge links antipodal vectors and
absorb their signs into the endpoints of the nontrivial links.
The resulting $k$ cyclic constraints have correlations of absolute
value at most $1/2$. They can all be realized: the matrix with diagonal
one, those correlations on the cycle edges, and zero on all other
pairs is symmetric diagonally dominant and positive semidefinite.
Complete every nontrivial clique to a regular simplex. There is no
excess.

If $k=2$, leave one bridge link uncompressed; one exists because
$\ell\ge3$. The three remaining cyclic constraints have absolute
correlations at most $1/2$, so the same diagonally dominant
three-by-three matrix realizes them. Choose correlation $-1/2$ on the
retained bridge, which has cost $4/3$. All other links have their exact
simplex costs. The total excess is $1/3$.

If $k=1$, let the nontrivial clique have order $t\ge3$. The remaining
links form a path of length $L\ge2$ between its two ports; $L=1$ would
repeat an edge already in the clique. Give the clique its regular
simplex, so the port correlation is $c=-1/(t-1)$.
For the connecting path, \eqref{eq:path-certificate} has
$\theta=\arccos((-1)^Lc)\le2\pi/3$. Its excess is at most
\[
 L\tan^2\frac{\pi}{3L}\le2\tan^2\frac\pi6=\frac23.
\]

Finally, if $k=0$, then $B$ is a cycle. For even $\ell$, use alternating
antipodal vectors, with no excess. For odd $\ell$, noncompleteness gives
$\ell\ge5$. A planar assignment with edge correlations
$-\cos(\pi/\ell)$ has excess
\[
 \ell\tan^2\frac{\pi}{2\ell}
 \le5\tan^2\frac\pi{10}=5-2\sqrt5<\frac23.\qedhere
\]
\end{proof}

Call a vertex $w$ of a $2$-connected graph $B$ \emph{exceptional} if
$B-w$ is a block graph.

\begin{lemma}\label{lem:two-exceptional}
Let $B$ be a noncomplete $2$-connected graph with two distinct
exceptional vertices $u,v$. Let $C_1,\ldots,C_h$ be the components of
$B-\{u,v\}$, and put $
 P_i=B[C_i\cup\{u,v\}]-uv.$
Every $P_i$ is either an ordinary $u$--$v$ clique-chain or a graph $T_s$.
If $e=1$ when $uv\in E(B)$ and $e=0$ otherwise, then
\begin{equation}\label{eq:parallel-q}
 q(B)=\sum_{i=1}^h q(P_i)+h-1+2e.
\end{equation}
\end{lemma}
\begin{proof}
Every component $C_i$ has a neighbour of both $u$ and $v$, since
otherwise $B-u$ or $B-v$ would be disconnected. Induced connected
subgraphs of block graphs are block graphs: after intersecting each
clique block with the chosen vertex set, any nonempty intersections
are still cliques and can meet only at the original cut vertices.
Thus $C_i$, $B[C_i\cup\{u\}]$, and $B[C_i\cup\{v\}]$ are block graphs.

Fix $C=C_i$. In $B[C\cup\{u\}]$, the vertex $u$ is not a cut vertex,
since deleting it leaves $C$ connected. Therefore $u$ belongs to one
clique block, and $N_C(u)$ is the vertex set of that block minus $u$.
If $|N_C(u)|\ge2$, this set must be an entire block of $C$.
Indeed, otherwise a larger clique block of $C$ and the block containing
$u$ would intersect in at least two vertices, which is impossible for
distinct blocks. The same applies to $v$.
Hence each terminal attaches either at one vertex of $C$, or to every
vertex of one block of $C$.

If both terminals attach to the same block $K$ of $C$ of order at least
two, then $C=K$: any part of $C$ outside $K$ would be separated from both
terminals by the cut vertex through which it attaches to $K$, contrary
to the $2$-connectivity of $B$. Consequently $P_i=T_{|C|}$.

In all other cases, adding the two terminals extends separate clique
blocks or attaches individual $K_2$ blocks, so $P_i$ is a block graph.
Any leaf block avoiding both terminals would have private vertices
separated from all of $B$ by its cut vertex, again a contradiction.
Neither terminal is a cut vertex of $P_i$, since deleting it leaves a
connected graph. Since $uv$ is not an edge of $P_i$, the terminals
belong to distinct blocks. Thus these terminal blocks are exactly the
two leaves of the block--cut tree, which is a path. This is an ordinary
clique-chain. When both attachments are the same singleton, the same
argument forces $C$ to be that singleton, giving the two-edge path.

Finally,
\[
 \sum_i|V(P_i)|=|V(B)|+2h-2,\qquad
 \sum_i|E(P_i)|=|E(B)|-e.
\]
Substitution into the definition of $q$ gives \eqref{eq:parallel-q}.
\end{proof}

\begin{lemma}\label{lem:exceptional-gap}
Every noncomplete $2$-connected graph with at least two exceptional
vertices satisfies $\defect(B)\ge1/3$.
\end{lemma}
\begin{proof}
Use Lemma~\ref{lem:two-exceptional}. If $h\ge2$, give all branches the
certificates from Lemma~\ref{lem:terminal-certificates}, using the same two
orthogonal terminal vectors. This is possible by putting the remaining
vector spaces of the branches in mutually orthogonal complements of
the terminal span. If $uv$ is an edge, its cost is two.
By \eqref{eq:parallel-q} and Lemma~\ref{lem:certificate},
\begin{equation}\label{eq:parallel-deficit}
 \defect(B)\ge h-1-\sum_{i=1}^h\eta(P_i).
\end{equation}
For $h\ge3$, the right side is at least
\[
 h-1-h(6-4\sqrt2)\ge12\sqrt2-16>\frac13.
\]
For $h=2$, if at least one branch is not a two-edge path, then
\[
 \eta(P_1)+\eta(P_2)
 \le6-4\sqrt2+\frac27<\frac23,
\]
so \eqref{eq:parallel-deficit} applies.
If both branches are two-edge paths, $B$ is $C_4$ or $K_4$ minus an
edge. For $C_4$, an antipodal bipartite assignment has cost four, while
$q(C_4)=5$. For $K_4$ minus an edge, give its two nonadjacent vertices
the same vector and use three simplex vectors in total. Every edge
then has correlation $-1/2$, giving cost $20/3=q(B)-1/3$.
It remains to treat $h=1$. If $uv\notin E(B)$, the ordinary-chain
possibility is excluded by $2$-connectivity, so $B=T_s$.
Give $u,v$ the same vector and use a regular simplex on $s+1$ vectors
in total. Every edge has correlation $-1/s$, and $
 q(T_s)-\Phi_{T_s}=\frac{s-1}{s+1}\ge\frac13.$
If $uv\in E(B)$, the thick-link possibility would make $B$ complete.
The remaining possibility is an ordinary clique-chain closed by $uv$,
which is a noncomplete cyclic clique-chain. Apply
Lemma~\ref{lem:cyclic-certificate}.
\end{proof}

\begin{theorem}\label{thm:block-gap}
For every noncomplete $2$-connected graph $B$,
\begin{equation}\label{eq:block-gap}
 \kappa(B)\le q(B)-\frac13.
\end{equation}
\end{theorem}
\begin{proof}
We induct on $r=|V(B)|$. There is no noncomplete $2$-connected graph
of order three. Suppose $r\ge4$ and the result holds for smaller
orders. If $B$ has at least two exceptional vertices, apply
Lemma~\ref{lem:exceptional-gap}. We may therefore assume that at most one
vertex is exceptional.

Write $m=|E(B)|$, $q=q(B)$, and
\[
 d_w=q(B-w)-\kappa(B-w),\qquad D=\sum_{w\in V(B)}d_w.
\]
Every $B-w$ is connected. If $w$ is not exceptional, $B-w$ has a
noncomplete block of smaller order. The induction hypothesis and
Lemma~\ref{lem:matrix-splitting} give $d_w\ge1/3$.
For every other $w$, Theorem~\ref{thm:LTZ} gives $d_w\ge0$. Hence
\begin{equation}\label{eq:D-lower}
 D\ge\frac{r-1}{3}.
\end{equation}

Fix $0\ne M\in\cD(B)$ and put
$S=S_B(M)$, $T=T(M)$, $x=\tr(M)/T\in[0,1]$.
Since $T=\tr(M)+2\sum_{uv\in E(B)}M_{uv}$, Cauchy--Schwarz over the
edges yields the first envelope
\begin{equation}\label{eq:flat-envelope}
 \frac{4S^2}{T}\le2m(1-x).
\end{equation}
For each $w$, let
$\sigma_w=\sum_{u\sim w}\sqrt{M_{uw}}$. Apply the definition of
$\kappa(B-w)$ to $M-w$ and sum the resulting inequalities
\[
 2(S-\sigma_w)\le\sqrt{\kappa(B-w)T(M-w)}.
\]
Using Cauchy--Schwarz and
\begin{align*}
 \sum_w(S-\sigma_w)&=(r-2)S,\\
 \sum_wq(B-w)&=(r-2)(q-1),\\
 \sum_wT(M-w)&=(r-2)T+\tr(M),
\end{align*}
we obtain the second envelope
\begin{equation}\label{eq:averaged-envelope}
 \frac{4S^2}{T}
 \le\left(q-1-\frac{D}{r-2}\right)
       \left(1+\frac{x}{r-2}\right).
\end{equation}
The first factor is nonnegative because it equals
$\sum_w\kappa(B-w)/(r-2)$.

Set
\[
 x_0=\frac{r-2/3}{2m},\qquad
 a=q-1-\frac{r-1}{3(r-2)}>0.
\]
If $x\ge x_0$, \eqref{eq:flat-envelope} gives
$4S^2/T\le q-1/3$.
If $x<x_0$, \eqref{eq:D-lower} and \eqref{eq:averaged-envelope} give
$4S^2/T\le a(1+x_0/(r-2))$.
The latter is strictly smaller than $q-1/3$. Indeed,
\begin{align*}
 &18m(r-2)^2
 \left(q-\frac13-a\left(1+\frac{x_0}{r-2}\right)\right)\\
 &\hspace{1em}=9r^3-21r^2+7r+2-18m(r-2)\\
 &\hspace{1em}\ge6r^2+7r-34>0,
\end{align*}
where we used $m\le r(r-1)/2-1$, since $B$ is noncomplete.
Taking the supremum over $M$ proves the induction step.
\end{proof}

For two matrices $A=(a_{ij})$ and $B=(b_{ij})$ of the same order,
their \emph{Schur product} (or \emph{Hadamard product}) is defined by
\[
A\circ B=(a_{ij}b_{ij}).
\]
We now apply the matrix deficit to the negative spectral part of the
adjacency matrix. This gives a lower bound on positive square energy.

\begin{lemma}\label{lem:spectral-budget}
Let $G$ be a connected graph of order at least two. There are numbers
$\varepsilon_B\ge0$, one for each block, such that
\begin{equation}\label{eq:spectral-budget}
 \rho(G)\ge\sum_{B\in\cB(G)}\varepsilon_B.
\end{equation}
Every noncomplete block satisfies $\varepsilon_B\ge1/3$.
\end{lemma}
\begin{proof}
Put $A=A(G)$. Apply the nonedge-folding construction of Liu, Tang,
and Zhang~\cite[proof of Theorem~2.1]{LTZ} to the doubly nonnegative
matrix $A_-\circ A_-$, which is positive semidefinite by the Schur
product theorem. This gives
\begin{equation}\label{eq:folding}
 M=A_-\circ A_-+
 \sum_{\substack{\{u,v\}\subseteq V(G)\\uv\notin E(G)}}
 (A_-)_{uv}^2(e_u-e_v)(e_u-e_v)^{\mathsf T}.
\end{equation}
The sum is over distinct unordered pairs. Then $M\in\cD(G)$,
$M_{uv}=(A_-)_{uv}^2$ on edges, and
\[
 T(M)=\one^{\mathsf T}(A_-\circ A_-)\one=s^-(G).
\]
Split $M$ into the matrices $M_B$ from
Lemma~\ref{lem:matrix-splitting}. Define
\begin{equation}\label{eq:block-bc}
 T_B=T(M_B),\qquad
 b_B=-2\sum_{uv\in E(B)}(A_-)_{uv},\qquad
 c_B=\begin{cases}b_B^2/T_B,&T_B>0,\\0,&T_B=0.\end{cases}
\end{equation}
If $T_B=0$, then $M_B=0$ and $b_B=0$.
Since $AA_-=-A_-^2$,
\[
 \sum_Bb_B=s^-(G),\qquad \sum_BT_B=s^-(G).
\]
Cauchy--Schwarz implies
$s^-(G)\le\sum_Bc_B$.
Moreover,
\[
 |b_B|\le2\sum_{uv\in E(B)}|(A_-)_{uv}|=2S_B(M_B),
\]
so $c_B\le\kappa(B)$.
Let $\varepsilon_B=q(B)-c_B$.
Using $q(G)=\sum_Bq(B)$, we conclude that
\[
 \rho(G)=q(G)-s^-(G)\ge\sum_B\varepsilon_B,
 \qquad
 \varepsilon_B\ge\defect(B)\ge0.
\]
For noncomplete blocks, apply Theorem~\ref{thm:block-gap}.
\end{proof}

\begin{lemma}\label{lem:leaf-triangle}
For the quantities in Lemma~\ref{lem:spectral-budget}, every leaf block
$B\cong K_3$ satisfies
\begin{equation}\label{eq:leaf-triangle-deficit}
 \varepsilon_B>\frac13.
\end{equation}
\end{lemma}
\begin{proof}
Let $u$ be the cut vertex of $B$, and let $a,b$ be its private vertices.
Their interchange is an automorphism, and $e_a-e_b$ is an eigenvector
of $A(G)$ with eigenvalue $-1$.
Consequently, for some real $x,y$,
\[
 (A_-)_{aa}=(A_-)_{bb}=1-x,\qquad
 (A_-)_{ab}=-x,\qquad
 (A_-)_{ua}=(A_-)_{ub}=-y.
\]
The principal matrix on $\{a,b\}$ is positive semidefinite, so
$x\le1/2$. Since $a$ has exactly the two neighbours $u,b$,
$A_-^2=-AA_-$ gives
\[
 (1-x)^2+x^2+y^2\le x+y,
 \qquad (1-x)(1-2x)\le y(1-y).
\]
It follows that
\begin{equation}\label{eq:triangle-xy}
 0<x\le\frac12,\qquad 0\le y\le1.
\end{equation}
Indeed, $x\le0$ would make the left side of the second inequality at
least one, whereas $y(1-y)\le1/4$; its nonnegativity then forces
$y\in[0,1]$.
In the folded matrix \eqref{eq:folding}, the private diagonal entries
are $
 d=x+y-x^2-y^2, $
because $(A_-^2)_{aa}=x+y$ and the only off-diagonal edge entries in
that row have squares $x^2$ and $y^2$.
They are unchanged by splitting. The $a,b$ entry of $M_B$ is $x^2$,
and its two root-to-private entries are $y^2$.
Put $h=d+x^2=x+y-y^2>0$.
Positive semidefiniteness, restricted to the root and the vector
$(e_a+e_b)/\sqrt2$, forces the root diagonal of $M_B$ to be at least
$2y^4/h$. Therefore
\[
 T_B\ge2h+4y^2+\frac{2y^4}{h}
 =\frac{2(x+y)^2}{h}.
\]
Since $b_B=2(x+2y)$,
\begin{equation}\label{eq:triangle-F}
 c_B\le2F(x,y),\qquad
 F(x,y)=\frac{(x+2y)^2(x+y-y^2)}{(x+y)^2}.
\end{equation}
For $x>0$, $y\ge0$,
\[
 \frac{\partial F}{\partial x}
 =\frac{(x+2y)(x^2+xy+2y^3)}{(x+y)^3}\ge0.
\]
Thus $F(x,y)\le F(1/2,y)$, while
\begin{align*}
 \frac{11}{6}-F\!\left(\frac12,y\right)
 &=\frac{48y^4-24y^3-23y^2+7y+4}{3(2y+1)^2}>0,\\
 48y^4-24y^3-23y^2+7y+4
 &=\left(y-\frac23\right)^2(48y^2+40y+9)+\frac{11}{9}y.
\end{align*}
The final expression is positive for all $y\ge0$.
Hence $c_B<11/3$, and $\varepsilon_B=q(K_3)-c_B=4-c_B>1/3$.
\end{proof}

\section{Proof of the main theorem}\label{sec:main-proof}
\begin{proof}[Proof of Theorem~\ref{thm:main}]
Let $G$ be connected, with $\delta(G)\ge2$, and suppose that $G$ is not
a cycle. If $G$ is $2$-connected, the assertion follows from
Akbari, Hu, and Liu~\cite{AHL}.
We may therefore assume that $G$ has a cut vertex.
Every leaf block then has order at least three, since a leaf block
$K_2$ would contain a vertex of degree one.

Suppose first that some leaf block is a complete graph $K_t$ with
$t\ge4$. Choose three private vertices of that block and call their
set $X$. They induce a triangle, and $G-X$ is connected.
By Lemma~\ref{lem:variational} and Theorem~\ref{thm:LTZ},
\[
 s^+(G)\ge s^+(G[X])+s^+(G-X)
 \ge4+(n-3-1)=n.
\]
Henceforth every complete leaf block can be assumed to be a triangle.

If there are at least three leaf blocks, each has deficit at least
$1/3$: use Theorem~\ref{thm:block-gap} and Lemma~\ref{lem:spectral-budget} for noncomplete
blocks and Lemma~\ref{lem:leaf-triangle} for triangles.
Thus \eqref{eq:spectral-budget} gives $\rho(G)\ge1$.

There remain exactly two leaf blocks. The block--cut tree is therefore
a path. If an end block $B$ is a triangle, deleting all three vertices
of $B$ leaves a connected graph: its cut vertex belongs to exactly one
other block, and deletion of that vertex leaves that other block
connected (or leaves the single remaining vertex of a bridge).
All subsequent blocks remain attached in a chain. Consequently,
\[
 s^+(G)\ge s^+(B)+s^+(G-V(B))
 \ge4+(n-3-1)=n.
\]
We may thus assume that both end blocks are noncomplete.

If there are at least three noncomplete blocks, then
Lemma~\ref{lem:spectral-budget} and Theorem~\ref{thm:block-gap} again give $\rho(G)\ge1$.
Otherwise the two end blocks are the only noncomplete blocks, and
every internal block is complete, with bridges allowed.

Choose any cut vertex $v$. Since the block--cut tree is a path,
$G-v$ has exactly two components. Adjoin $v$ to each component and
let the resulting induced connected graphs be $H_1,H_2$.
Each $H_i$ consists of one end block followed by a chain of complete
blocks and bridges, and $H_i-v$ is connected.
By Lemma~\ref{lem:block-chain}, $ \rho(H_1)\ge\frac34$ and $\rho(H_2)\ge\frac34.
$
Applying \eqref{eq:coalescence-linear} at $v$ yields
\[
 \rho(G)\ge\frac89\left(\frac34+\frac34\right)-\frac29
 =\frac{10}{9}>1.
\]
This completes the proof in every case.
\end{proof}

\begin{samepage}
Combining the theorem with Lemma~\ref{lem:cycles} identifies the actual
exceptions within the minimum-degree-two class.
\begin{corollary}\label{cor:exceptions}
If $G$ is a connected graph with $\delta(G)\ge2$, then
\(
 s^+(G)<|V(G)|\) if and only if
$ G\cong C_{4k+1}$, for some integer $k\ge1.$
\end{corollary}
\begin{proof}
Noncycles are covered by Theorem~\ref{thm:main}. For cycles, the assertion
follows directly from \eqref{eq:cycles}.
\end{proof}
\end{samepage}

Every $2$-edge-connected graph of order at least three has minimum degree
at least two. Thus Theorem~\ref{thm:main} also gives the following consequence.

\begin{corollary}
Let $G$ be a $2$-edge-connected graph of order $n\ge3$.
If $G$ is not a cycle, then $s^+(G)\ge n$.
\end{corollary}

\section*{Acknowledgments}

Supported by National Natural Science Foundation of China (12671396, 12331012).

\section{Declaration of AI Usage}

The authors used an artificial intelligence tool during the preparation of this paper. The tool was employed to improve the English language and to assist with certain computational tasks. All results, derivations, and conclusions were independently verified by the authors. The authors accept full responsibility for the correctness of the  manuscript.

\end{document}